\documentclass{amsart}
\usepackage[ocgcolorlinks,linktoc=all]{hyperref}
\hypersetup{citecolor=blue,linkcolor=red}
\newtheorem{theorem}{Theorem}

\newtheorem*{thmmain}{Theorem}
\newtheorem{lemma}[theorem]{Lemma}

\theoremstyle{definition}

\theoremstyle{remark}
\newtheorem{remark}[theorem]{Remark}

\numberwithin{equation}{section}

\begin{document}
\title[Gauss curvature flow]
 {A note on the Gauss curvature flow}
\author[M.N. Ivaki]{Mohammad N. Ivaki}
\address{Institut f\"{u}r Diskrete Mathematik und Geometrie, Technische Universit\"{a}t Wien,
Wiedner Hauptstr. 8--10, 1040 Wien, Austria}
\curraddr{}
\email{mohammad.ivaki@tuwien.ac.at}

\dedicatory{}
\begin{abstract}
Using polar convex bodies and the $C_0$-bounds from Guan and Ni \cite{PL}, we obtain a uniform lower bound on the Gauss curvature of the normalized solution of the Gauss curvature flow without using Chow's Harnack inequality \cite{Ch2}.
\end{abstract}
\maketitle
\section{Introduction}
Firey \cite{W} introduced the Gauss curvature flow to model the changing shape of a tumbling stone subjected to collisions from all directions with uniform frequency. Assuming the existence, uniqueness, and regularity of the solutions (settled later by K.S. Chou (K. Tso) \cite{cho}), Firey proved that if the initial convex surface is origin-symmetric, the solution to the flow converges to the origin in finite time and the normalized solution, having a fixed volume equal to the volume of the unit ball, converges in the Hausdorff distance to the unit ball. Furthermore, Firey conjectured that the same conclusion must hold if one begins the flow from any convex surface. Andrews gave an affirmative answer to this question through an elegant approach \cite{BA1}. In \cite{BA1}, he applied the parabolic maximum principle to the difference of principal curvatures and used Chow's Harnack inequality to obtain regularity of the normalized solution and the asymptotic roundness. Recently, Guan and Ni \cite{PL} studied some fine properties of Firey's entropy functional and used Chow's Harnack inequality to resolve the long-standing issue with the regularity of normalized solutions in higher dimensions. The issue involves obtaining a lower bound on the Gauss curvature of the normalized solution without imposing any condition on the initially smooth, strictly convex hypersurface. In this paper, we show that by using polar convex bodies, it is possible to avoid the Harnack inequality in the process of obtaining a uniform lower bound. The uniform upper bound on the Gauss curvature of the normalized solution, in view of Tso's trick \cite{cho}, is fairly easy and has been known for a long time (see \cite[Theorem 5.1]{PL} for details). Our approach may prove useful for geometric flows when no improving pinching estimates or Harnack inequalities are known to exist or when the flow is not invariant under Euclidean translations. We consider the evolution of polar bodies and apply the maximum principle to the difference of a suitable power of the Euclidean norm of ``polar embedding" and the speed of the ``dual flow." We remark that in the presence of an improving pinching estimate, one may circumvent the need for Harnack estimates by applying the method developed by Andrews and McCoy \cite{AJ}. See \cite{PL} for many useful comments on earlier works by Andrews, Chow, and Hamilton on the Gauss curvature flow \cite{BA,BA1,Ch2,Ha}.
\section{Background material}
The setting of this paper is $n$-dimensional Euclidean space. Let $\mathbb{S}^{n-1}$ denote the unit sphere, centered at the origin, in $\mathbb{R}^n.$ Let $\bar{g}_{ij}$ and $\bar{\nabla}$ denote, respectively, the standard metric and the Levi-Civita connection of $\mathbb{S}^{n-1}$. A compact convex subset of $\mathbb{R}^{n}$ with non-empty interior is called a \emph{convex body}. Write $\mathcal{F}^n$ for the set of smooth, strictly convex bodies in $\mathbb{R}^n.$

For a convex body $K\in\mathcal{F}^n$ let $\partial K$ denote its boundary and $\nu:\partial K\to \mathbb{S}^{n-1}$ denote its Gauss map; that is, at each point $x\in\partial K$, $\nu(x)$ is the unit outward normal at $x$. Let $\varphi:\partial K\to\mathbb{R}^n$ be a smooth embedding of $\partial K$. The support function of $K$ with the origin in its interior, $s_K:\mathbb{S}^{n-1}\to[0,\infty)$, can be defined for $x\in\partial K$ by
\[s_K(\nu(x)):= \langle \varphi(x), \nu(x) \rangle.\]
The Gauss curvature of $\partial K$, $\mathcal{K}$, is related to the support function of the convex body by \[\sigma_{n-1}:=\frac{1}{\mathcal{K}\circ\nu^{-1}}:=\frac{\det(\mathfrak{r}_{ij}:=\bar{\nabla}_i\bar{\nabla}_js+\bar{g}_{ij}s)}{\det \bar{g}_{ij}}.\]

Let $\varphi_0$ be a smooth, strictly convex embedding of $\partial K$. A family of convex bodies $\{K_t\}_t\subset \mathcal{F}^n$ given by the smooth embeddings $\varphi:\partial K\times [0,T)\to \mathbb{R}^n$ is said to be a solution of the Gauss curvature flow if $\varphi$ satisfies the initial value problem
\begin{equation}\label{e: flow0}
 \partial_{t}\varphi(x,t)=-\mathcal{K}(x,t)\, \nu(x,t),~~
 \varphi(\cdot,0)=\varphi_{0}(\cdot).
\end{equation}
In this equation, $\mathcal{K}(x,t)$ is the Gauss curvature of $\partial K_t:=\varphi(\partial K,t)$ at the point where the outer unit normal is $\nu(x,t)$, and $T$ is the maximal time that the flow exists. Tso \cite{cho} proved that $K_t$ shrink to a point in a finite time $T$. Throughout this paper, we assume that this point is the origin of $\mathbb{R}^n.$
It is easy to see that support functions of $\{K_t\}_{t\in [0,T)}$ satisfy
\begin{equation}\label{e: ev support}
 \partial_{t}s(z,t)=-\frac{1}{\sigma_{n-1}(z,t)},~~
 s(\cdot ,t)=s_{K_t}(\cdot).
\end{equation}

\section{Evolution equation of polar bodies}
The polar body, $K^{\ast}$, of the convex body $K$ with the origin in its interior is the convex body that is defined as
\[K^{\ast}=\{x\in\mathbb{R}^n: \langle x,y\rangle\leq 1 \mbox{~for~all~}y\in K\}.\]
When we are dealing with quantities associated with $K^{\ast}$ we will indicate this with a super-script $~^{\ast}.$

We want to obtain the evolution equation of $s_{K_t^{\ast}}$. To do so, we first parameterize $\partial K_t$ over the unit sphere
\[\varphi=r(z(\cdot,t),t)z(\cdot,t):\mathbb{S}^{n-1}\to\mathbb{R}^{n},\]
where $r(z(\cdot,t),t)$ is the radial function of $K_t$ in the direction $z(\cdot,t).$

Let $K$ be a smooth, strictly convex body with the origin in its interior. Suppose $\partial K$ is parameterized with the radial function $r.$
The metric $[g_{ij}]_{1\leq i,j\leq n-1}$, unit normal $\nu$, support function $s$, and the second fundamental form $[h_{ij}]_{1\leq i,j\leq n-1}$ of $\partial K$
can be written in terms of $r$ and its partial derivatives as follows:
\begin{description}
  \item[a] $\displaystyle g_{ij}=r^2\bar{g}_{ij}+\bar{\nabla}_ir\bar{\nabla}_jr,$
  \item[b] $\displaystyle \nu=\frac{r z-\bar{\nabla}r}{\sqrt{r^2+\|\bar{\nabla}r\|_{\bar{g}}^2}},$
  \item[c] $\displaystyle s=\frac{r^2}{\sqrt{r^2+\|\bar{\nabla}r\|_{\bar{g}}^2}},$
  \item[d] $\displaystyle h_{ij}=\frac{-r\bar{\nabla}_i\bar{\nabla}_jr+2\bar{\nabla}_ir\bar{\nabla}_jr+
  r^2\bar{g}_{ij}}{\sqrt{r^2+\|\bar{\nabla}r\|_{\bar{g}}^2}}.$
\end{description}
Since $\frac{1}{r}$ is the support function of $\partial K^{\ast}$, we can calculate the entries of $[\mathfrak{r}^{\ast}_{ij}]_{1\leq i,j\leq n-1}$:
\[\mathfrak{r}^{\ast}_{ij}=\bar{\nabla}_i\bar{\nabla}_j\frac{1}{r}+\frac{1}{r}\bar{g}_{ij}=
\frac{-r\bar{\nabla}^2_{ij}r+2\bar{\nabla}_ir\bar{\nabla}_jr+r^2\bar{g}_{ij}}{r^3}.\]
Thus, using (\textbf{d}) we get
\begin{equation}\label{e: h and h polar}
\mathfrak{r}^{\ast}_{ij}=\frac{\sqrt{r^2+\|\bar{\nabla}r\|_{\bar{g}}^2}}{r^3}h_{ij}.
\end{equation}
\begin{lemma}
As $K_t$ evolve according to (\ref{e: ev support}), their radial functions evolve as follows
\begin{align}\label{e: ev r}
\partial_t r=-\frac{\sqrt{r^2+\|\bar{\nabla}r\|_{\bar{g}}^2}}{r}\mathcal{K}.
\end{align}
\end{lemma}
\begin{proof}
\begin{align*}
\partial_t \varphi&=\partial _t \left(r(z(\cdot,t),t)z(\cdot,t)\right)\\
&=(\partial _t r)z+\langle\bar{\nabla}r,\partial_t z\rangle z+r\partial _t z\\
&=-\mathcal{K}\nu\\
&=-\mathcal{K}\frac{r z-\bar{\nabla}r}{\sqrt{r^2+\|\bar{\nabla}r\|_{\bar{g}}^2}},
\end{align*}
where from the third line to the fourth line we used (\textbf{b}).
By comparing the terms on the second line with the terms on the fourth line we obtain
\begin{align}\label{e: tan com}
r\partial_t z=\frac{\mathcal{K}\bar{\nabla}r}{\sqrt{r^2+\|\bar{\nabla}r\|_{\bar{g}}^2}},
\end{align}
and
\begin{equation}\label{e: nor com}
\partial _t r+\langle\bar{\nabla}r,\partial_t z\rangle =-\frac{\mathcal{K}r}{\sqrt{r^2+\|\bar{\nabla}r\|_{\bar{g}}^2}}.
\end{equation}
Replacing $\partial_t z$ in (\ref{e: nor com}) by its equal expression (e.q. equation (\ref{e: tan com})) completes the proof.
\end{proof}
\begin{theorem} \label{thm: main}
As $\{K_t\}_{t\in [0,T)}$ evolves according to evolution equation (\ref{e: ev support}), the family of convex bodies $\{K^{\ast}_t:=(K_t)^{\ast}\}_{t\in [0,T)}$ evolves by
$$s^{\ast}:\mathbb{S}^{n-1}\times[0,T)\to\mathbb{R}$$
\begin{align*}
\partial_t s^{\ast}(\cdot,t)=\frac{s^{\ast n+2}\sigma_{n-1}^{\ast}}{\left(s^{\ast 2}+\|\bar{\nabla}s^{\ast}\|_{\bar{g}}^2\right)^{\frac{n}{2}}}(\cdot,t),~~ s^{\ast}(\cdot,t)=s_{K_t^{\ast}}(\cdot).
\end{align*}
\end{theorem}
\begin{proof}
Using the identities
\[\mathcal{K}=\frac{\det h_{ij}}{\det g_{ij}},~~
 \frac{1}{\sigma_{n-1}^{\ast}}=\frac{\det \bar{g}_{ij}}{\det \mathfrak{r}^{\ast}_{ij}},~~
\frac{\det \bar{g}_{ij}}{\det g_{ij}}=\frac{1}{r^{2n-4}(r^2+\|\bar{\nabla}r\|_{\bar{g}}^2)},\]
 equation (\textbf{c}), (\ref{e: h and h polar}), and evolution equation (\ref{e: ev r}), we calculate
\begin{align*}
&\partial_t s^{\ast}=\partial_t \frac{1}{r}\\
&=\frac{\sqrt{r^2+\|\bar{\nabla}r\|_{\bar{g}}^2}}{r^3}
\mathcal{K}\\
&=\frac{\sqrt{r^2+\|\bar{\nabla}r\|_{\bar{g}}^2}}{r^3}\frac{\det h_{ij}}{\det g_{ij}}\\
&=\frac{\sqrt{r^2+\|\bar{\nabla}r\|_{\bar{g}}^2}}{r^3}\frac{\det \mathfrak{r}^{\ast}_{ij}}{\det \bar{g}_{ij}}
\frac{\det \bar{g}_{ij}}{\det g_{ij}}
\frac{\det h_{ij}}{\det \mathfrak{r}^{\ast}_{ij}}\\
&=\left(\frac{\sqrt{r^2+\|\bar{\nabla}r\|_{\bar{g}}^2}}{r^3}\right)^{2-n}\sigma_{n-1}^{\ast}
\left(r^{2n-4}(r^2+\|\bar{\nabla}r\|_{\bar{g}}^2)\right)^{-1}.\\
\end{align*}
Replacing $ r$ by $\frac{1}{s^{\ast}}$ in the expression
$\left(\frac{\sqrt{r^2+\|\bar{\nabla}r\|_{\bar{g}}^2}}{r^3}\right)^{2-n}
\left(r^{2n-4}(r^2+\|\bar{\nabla}r\|_{\bar{g}}^2)\right)^{-1}$ completes the proof.
\end{proof}
Theorem \ref{thm: main} implies that there is a unique solution $\varphi^{\ast}:\partial K\times [0,T)\to \mathbb{R}^{n}$ of the flow equation
\begin{align}\label{ref: main polar}
\partial_t \varphi^{\ast}(\cdot,t)=\left(\frac{\langle \varphi^{\ast},\nu^{\ast}\rangle^{n+2}}{\|\varphi^{\ast}\|^n}\frac{1}{\mathcal{K}^{\ast}}\right)(\cdot,t)\nu^{\ast}(\cdot,t), \varphi^{\ast}(\partial K,0)=\partial K_{0}^{\ast}
\end{align}
that satisfies $\varphi^{\ast}(\partial K,t)=\partial K_{t}^{\ast}$ (see Urbas \cite{Urbas} for details).
\section{uniform lower bound on the Gauss curvature}
\begin{lemma}
We have the following evolution equations under flow (\ref{ref: main polar}):
  \begin{align*}
  \partial_t \|\varphi^{\ast}\|^{n+1}=&\frac{\langle \varphi^{\ast},\nu^{\ast}\rangle^{n+2}}{\|\varphi^{\ast}\|^n}\frac{\dot{\mathcal{K}}_{~j}^{\ast i}}{\mathcal{K}^{\ast2}}\nabla_i\nabla^j \|\varphi^{\ast}\|^{n+1}\\
  &-(n^2-1)\frac{\langle \varphi^{\ast},\nu^{\ast}\rangle^{n+2}}{\|\varphi^{\ast}\|^3}\frac{\dot{\mathcal{K}}^{\ast ij}}{\mathcal{K}^{\ast2}}\langle \varphi^{\ast}, \varphi^{\ast}_{,i}\rangle\langle \varphi^{\ast}, \varphi^{\ast}_{,j}\rangle\\
  &+\frac{(n^2+n)}{\mathcal{K}^{\ast}}\frac{\langle \varphi^{\ast},\nu^{\ast}\rangle^{n+3}}{\|\varphi^{\ast}\|}-(n+1)\frac{\langle \varphi^{\ast},\nu^{\ast}\rangle^{n+2}}{\|\varphi^{\ast}\|}\frac{\dot{\mathcal{K}}_{~j}^{\ast i}}{\mathcal{K}^{\ast2}}\delta_{i}^{j},
  \end{align*}
  \begin{align*}&\partial_t \left(\frac{\langle \varphi^{\ast},\nu^{\ast}\rangle^{n+2}}{\|\varphi^{\ast}\|^n}\frac{1}{\mathcal{K}^{\ast}}\right)=\frac{\langle \varphi^{\ast},\nu^{\ast}\rangle^{n+2}}{\|\varphi^{\ast}\|^n}\frac{\dot{\mathcal{K}}_{~j}^{\ast i}}{\mathcal{K}^{\ast2}}\nabla_i\nabla^j \left(\frac{\langle \varphi^{\ast},\nu^{\ast}\rangle^{n+2}}{\|\varphi^{\ast}\|^n}\frac{1}{\mathcal{K}^{\ast}}\right)\\
      &+\frac{1}{\mathcal{K}^{\ast3}}\frac{\langle \varphi^{\ast},\nu^{\ast}\rangle^{2n+4}}{\|\varphi^{\ast}\|^{2n}}\dot{\mathcal{K}}_{~j}^{\ast i}h_{~i}^{\ast k}h_{~k}^{\ast j}-\frac{(n+2)}{\mathcal{K}^{\ast}}\frac{\langle \varphi^{\ast},\nu^{\ast}\rangle^{n+1}}{\|\varphi^{\ast}\|^n}\langle \varphi^{\ast},T\varphi^{\ast}\left(\nabla\left(\frac{\langle \varphi^{\ast},\nu^{\ast}\rangle^{n+2}}{\|\varphi^{\ast}\|^n}\frac{1}{\mathcal{K}^{\ast}}\right)\right)\rangle\\
      &+\frac{(n+2)}{\mathcal{K}^{\ast2}}\frac{\langle \varphi^{\ast},\nu^{\ast}\rangle^{2n+3}}{\|\varphi^{\ast}\|^{2n}}-\frac{n}{\mathcal{K}^{\ast2}}\frac{\langle \varphi^{\ast},\nu^{\ast}\rangle^{2n+5}}{\|\varphi^{\ast}\|^{2n+2}}.
      \end{align*}
\end{lemma}
\begin{proof}
The evolution equation of $\|\varphi^{\ast}\|^{n+1}$ is straightforward. To obtain the second evolution equation, note that under flow (\ref{ref: main polar}), the metric $[g^{\ast}_{ij}]_{1\le i,j\le n-1}$, Weingarten tensor $[h_{~i}^{\ast j}=h^{\ast}_{ik}g^{\ast kj}]_{1\leq i,j\leq n-1}$, and the normal vector $\nu^{\ast}$ evolve by
\[\partial_t g^{\ast}_{ij}=2\frac{\langle \varphi^{\ast},\nu^{\ast}\rangle^{n+2}}{\|\varphi^{\ast}\|^n}\frac{h^{\ast}_{ij}}{\mathcal{K}^{\ast}},\]
\[\partial_t h_{~i}^{\ast j}=-\nabla_i\nabla^j\left(\frac{\langle \varphi^{\ast},\nu^{\ast}\rangle^{n+2}}{\|\varphi^{\ast}\|^n}\frac{1}{\mathcal{K}^{\ast}}\right)-\left(\frac{\langle \varphi^{\ast},\nu^{\ast}\rangle^{n+2}}{\|\varphi^{\ast}\|^n}\frac{1}{\mathcal{K}^{\ast}}\right)h_{~i}^{\ast k}h_{~k}^{\ast j},\]
\[\partial_t\nu^{\ast}=-T\varphi^{\ast}\left(\nabla\left(\frac{\langle \varphi^{\ast},\nu^{\ast}\rangle^{n+2}}{\|\varphi^{\ast}\|^n}\frac{1}{\mathcal{K}^{\ast}}\right)\right).\]
See Huisken \cite[Theorem 3.4]{Hus} for details.
\end{proof}
\begin{lemma} \label{lem: prev}
If $\gamma\|\varphi^{\ast}\|\leq\langle \varphi^{\ast},\nu^{\ast}\rangle$ for some $0<\gamma\leq 1$ on $[0,T)$, then we can find $\lambda\geq\left(\frac{2n^2+5n+2}{(n-1)\gamma^{\frac{4n+8}{n-1}+2n+4}}\right)^{n-1}$ large enough such that
$\|\varphi^{\ast}\|^{n-1}\leq \frac{\lambda}{\mathcal{K}^{\ast}}$
on $[0,T)$.
\end{lemma}
\begin{proof} We let $\lambda\geq\left(\frac{2n^2+5n+2}{(n-1)\gamma^{\frac{4n+8}{n-1}+2n+4}}\right)^{n-1}$ be a large number such that $\chi:=\|\varphi^{\ast}\|^{n+1}-\lambda \frac{\langle \varphi^{\ast},\nu^{\ast}\rangle^{n+2}}{\|\varphi^{\ast}\|^n}\frac{1}{\mathcal{K}^{\ast}}$ is negative at $t=0$. Our aim is to show that $\chi$ remains negative. We calculate the evolution equation of $\chi$ and apply the maximum principle to $\chi$ on $[0,\tau]$, where $\tau>0$ is the first time that for some $y\in\partial K_{\tau}^{\ast}$ we have $\chi(\tau,y)=0$. Note that at such a point where the maximum of $\chi$ is achieved
\begin{align*}\nabla \chi=0\Rightarrow\langle \varphi^{\ast},T\varphi^{\ast}\left(\nabla\left(\lambda\frac{\langle \varphi^{\ast},\nu^{\ast}\rangle^{n+2}}{\|\varphi^{\ast}\|^n}\frac{1}{\mathcal{K}^{\ast}}\right)\right)\rangle&=\langle \varphi^{\ast},T\varphi^{\ast}\left(\nabla\|\varphi^{\ast}\|^{n+1}\right)\rangle\\
&=(n+1)\|\varphi^{\ast}\|^{n-1}\|\varphi^{\ast\top}\|^2,
\end{align*}
 $$\frac{\langle \varphi^{\ast},\nu^{\ast}\rangle^{n+2}}{\|\varphi^{\ast}\|^n}\frac{\dot{\mathcal{K}}_{~j}^{\ast i}}{\mathcal{K}^{\ast2}}\nabla_i\nabla^j \chi\leq0,$$
 and in view of the assumption $\gamma\|\varphi^{\ast}\|\leq\langle \varphi^{\ast},\nu^{\ast}\rangle$ we have
 $$\gamma^{\frac{n+2}{n-1}}\left(\frac{\lambda}{\mathcal{K}^{\ast}}\right)^{\frac{1}{n-1}}\leq\|\varphi^{\ast}\|\leq \left(\frac{\lambda}{\mathcal{K}^{\ast}}\right)^{\frac{1}{n-1}}.$$
Here $\varphi^{\ast\top}(\cdot,t)$ is the tangential component of $\varphi^{\ast}(\cdot,t)$ to $\partial K_t^{\ast}.$ Moreover, we always have
$$-(n^2-1)\frac{\langle \varphi^{\ast},\nu^{\ast}\rangle^{n+2}}{\|\varphi^{\ast}\|^3}\frac{\dot{\mathcal{K}}^{\ast ij}}{\mathcal{K}^{\ast2}}\langle \varphi^{\ast}, \varphi^{\ast}_{,i}\rangle\langle \varphi^{\ast}, \varphi^{\ast}_{,j}\rangle=-\frac{n^2-1}{4}\frac{\langle \varphi^{\ast},\nu^{\ast}\rangle^{n+2}}{\|\varphi^{\ast}\|^3}\frac{\dot{\mathcal{K}}^{\ast ij}}{\mathcal{K}^{\ast2}}(\|\varphi^{\ast}\|^2)_{,i}(\|\varphi^{\ast}\|^2)_{,j}\le 0.$$
Thus, at $(\tau,y)$ we obtain
\begin{align*}
0\leq\partial_t\chi\leq& \frac{(n^2+n)}{\mathcal{K}^{\ast}}\frac{\langle \varphi^{\ast},\nu^{\ast}\rangle^{n+3}}{\|\varphi^{\ast}\|}-(n+1)\frac{\langle \varphi^{\ast},\nu^{\ast}\rangle^{n+2}}{\|\varphi^{\ast}\|}\frac{\dot{\mathcal{K}}_{~j}^{\ast i}}{\mathcal{K}^{\ast2}}\delta_{i}^{j}\\
&-\frac{\lambda}{\mathcal{K}^{\ast3}}\frac{\langle \varphi^{\ast},\nu^{\ast}\rangle^{2n+4}}{\|\varphi^{\ast}\|^{2n}}\dot{\mathcal{K}}_{~j}^{\ast i}h_{~i}^{\ast k}h_{~k}^{\ast j}+\frac{(n+2)(n+1)}{\mathcal{K}^{\ast}}\frac{\langle \varphi^{\ast},\nu^{\ast}\rangle^{n+1}}{\|\varphi^{\ast}\|}\|\varphi^{\ast\top}\|^2\\
      &-\lambda\frac{(n+2)}{\mathcal{K}^{\ast2}}\frac{\langle \varphi^{\ast},\nu^{\ast}\rangle^{2n+3}}{\|\varphi^{\ast}\|^{2n}}+\lambda\frac{n}{\mathcal{K}^{\ast2}}\frac{\langle \varphi^{\ast},\nu^{\ast}\rangle^{2n+5}}{\|\varphi^{\ast}\|^{2n+2}}.
\end{align*}
Dropping all the negative terms except $-\frac{\lambda}{\mathcal{K}^{\ast3}}\frac{\langle \varphi^{\ast},\nu^{\ast}\rangle^{2n+4}}{\|\varphi^{\ast}\|^{2n}}\dot{\mathcal{K}}_{~j}^{\ast i}h_{~i}^{\ast k}h_{~k}^{\ast j}$ and then
using the inverse-concavity of $\mathcal{K}^{\ast\frac{1}{n-1}}$ (e.q. $\dot{\mathcal{K}}_{~i}^{\ast j}h_{~k}^{\ast i}h_{~j}^{\ast k}\geq (n-1)\mathcal{K}^{\ast\frac{n}{n-1}}$),
the right-hand side of $\partial_t\chi$ simplifies to
\begin{align*}
0\leq\partial_t\chi
      <&
\frac{(n^2+n)}{\mathcal{K}^{\ast}}\|\varphi^{\ast}\|^{n+2}
-(n-1)\gamma^{2n+4}\frac{\lambda}{\mathcal{K}^{\ast\frac{2n-3}{n-1}}}\|\varphi^{\ast}\|^4\\
&+\frac{(n+2)(n+1)}{\mathcal{K}^{\ast}}\|\varphi^{\ast}\|^{n+2}
      +\lambda\frac{n}{\mathcal{K}^{\ast2}}\|\varphi^{\ast}\|^{3}\\
      \leq& \frac{\lambda^{\frac{n+2}{n-1}}}{\mathcal{K}^{\ast\frac{2n+1}{n-1}}}\left(2n^2+5n+2-(n-1)\gamma^{\frac{4n+8}{n-1}+2n+4}\lambda^{\frac{1}{n-1}}\right)<0.
\end{align*}
\end{proof}
\begin{thmmain}[Uniform lower bound]
There is a uniform lower bound on the Gauss curvature of the normalized solution of the Gauss curvature flow.
\end{thmmain}
\begin{proof}
In view of the $C_0$ estimates of Guan and Ni \cite[Estimates (5.8)]{PL} we have $$0< C_1\leq s^{\ast}(\cdot,t)(T-t)^{\frac{1}{n}}\leq C_2<+\infty.$$ Therefore, the identity $\|\varphi^{\ast}\|^2=s^{\ast2}+\|\bar{\nabla}s^{\ast}\|^2$ and Lemma \ref{lem: prev} with $\gamma =\frac{C_1}{C_2}$ imply that there is a uniform upper bound on $\mathcal{K}^{\ast}(T-t)^{-\frac{n-1}{n}}.$ To complete the proof, we recall the following identity. For every $x \in \partial K$, there exists an $x^\ast \in \partial K^\ast$ such that
\[\left(\frac{\mathcal{K}}{s^{n+1}}\right)(x)\left(\frac{\mathcal{K}^{\ast}}{s^{\ast n+1}}\right)(x^{\ast})=1,\]
and $x,x^{\ast}$ are related by $\langle x, x^{\ast}\rangle=1$ (proof of this identity is given in \cite{H,Kal}).
Using the $C_0$ estimates of Guan and Ni \cite[Estimates (5.8)]{PL} once again, we conclude that $\mathcal{K}(T-t)^{\frac{n-1}{n}}$ is uniformly bounded below.
\end{proof}
\begin{remark}
To my knowledge, a quantity similar to $\chi$ was first considered by O.C. Schn\"{u}rer \cite[Lemma 6.1]{Sch} and then by Q.R. Li in \cite[Lemma 5.1]{Li} to obtain regularity of the normalized solutions of the inverse Gauss curvature flow (in fact, O.C. Schn\"{u}rer considered several more curvature flows). Therefore, in view of dual flow (\ref{ref: main polar}) and the earlier works by the author and A. Stancu \cite{IS,Ivaki,Ivaki1}, it was very natural to consider $\chi$ as a candidate for the purpose of this paper.
\end{remark}
\textbf{Acknowledgment:}
The author would like to thank the referee for suggestions.
The work of the author was supported in part by Austrian Science Fund (FWF) Project P25515-N25.
\bibliographystyle{amsplain}

\end{document}